\documentclass{amsart}
\usepackage{amsfonts,amscd,amssymb,amsmath,amsthm}

\begin{document}
%{\center\small Mathematics Subject Classifications: 05C80, 05C30, 34E05, 60C05}
%{\center\small Keywords: enumeration, directed graphs, strongly connected, asymptotics}

\def\si{\par\smallskip\noindent}
\def\bi{\par\bigskip\noindent}
\def\pr{\text{ P\/}}
\def\ex{\text{E\/}}
\def\de{\delta}
\def\eps{\varepsilon}
\def\la{\lambda}
\def\a{\alpha}
\def\be{\beta}
\def\de{\Delta}
\def\sig{\sigma}
\def\ga{\gamma}
\def\part{\partial}
\def\tag#1 {\eqno(#1)}
\def\Cal{\mathcal}
\def\var{\text{Var\/}}

\newtheorem{Theorem}{Theorem}[section]
\newtheorem{theorem}{Theorem}[section]
\newtheorem{Lemma}{Lemma}[section]
\newtheorem{Proposition}{Proposition}[section]
\newtheorem{Corollary}{Corollary}[section]
\newtheorem{corollary}{Corollary}[section]
\numberwithin{equation}{section}

\newcommand{\R}{\mathbb{R}}
\newcommand{\Z}{\mathbb{Z}}
\newcommand{\prob}{\mathbb{P}}
\newcommand{\link}{\mbox{lk}}

\title[Inside the critical window for cohomology]{Inside the critical window for cohomology of random $k$-complexes}
\author{Matthew Kahle}
\author{Boris Pittel}
\thanks{The first author gratefully acknowledges DARPA grant \# N66001-12-1-4226}
\thanks{The second author gratefully acknowledges NSF grant \# DMS-1101237}

\maketitle

\begin{abstract}  We prove sharper versions of theorems of Linial--Meshulam and Meshulam--Wallach which describe the
behavior for $(\Z/2)$-cohomology of a random $k$-dimensional simplicial complex within a narrow
transition window. In particular, we show that within this window the Betti number $\beta^{k-1}$ is in the limit Poisson distributed. For $k=2$ we also prove that in an accompanying growth process, with high probability, first cohomology vanishes exactly at the moment when the last isolated $(k-1)$-simplex gets covered by a $k$-simplex.
\end{abstract}

\section{Introduction}

In 1959 Erd\H{o}s and R\'enyi pioneered a systematic study of a graph $G(n,m)$
chosen uniformly at random among all graphs on vertex set $[n] = \{ 1, 2, \dots, n \}$ with exactly
$m$ edges. They found the threshold value $\bar m$ for
connectedness of $G(n,m)$ \cite{er}.

Here and throughout the paper {\it ``with high probability (w.h.p)''} means that the probability of an event approaches $1$ as the number of vertices $n \to \infty$.

\begin{theorem} [Erd\H{o}s--R\'enyi] \label{thm:ER} If
$$
m=\frac{n}{2}(\log n +c),
$$
where $c \in \R$ is constant, then w.h.p.\ $G(n,m)$ consists of a giant component and isolated vertices, and the number of isolated vertices converges in distribution to Poisson with mean $e^{-c}$.  In particular
  $$\prob [ G(n,m) \mbox{ is connected}\, ] \to e^{-e^{-c}},$$ 
as $n \to \infty.$
of $G(n,m)$.
\end{theorem}

Consequently, $\bar{m}=(n/2)\log n$ is a sharp threshold for connectedness, in the following sense.

\begin{theorem} \label{cor:ER}  Let $\omega \to \infty$ arbitrarily slowly. If 
$$m=\frac{n}{2}(\log n + \omega),$$
then w.h.p.\ $G(n,m)$ is connected, and if
$$m=\frac{n}{2}(\log n - \omega),$$
then w.h.p.\ $G(n,m)$ is disconnected.
\end{theorem}

In 1969 Stepanov \cite{st} considered the 
Bernoulli counterpart $G(n,p)$, a random graph on $[n]$ such that a pair
$(i,j)$ forms an edge with probability $p$ independently of all other pairs.
He determined the threshold value $\bar p$ for connectedness of $G(n,p)$,
and  $\bar m \sim\bar p\binom{n}{2}$.  Informally,
this had to be expected because $G(n,m)$ is distributed as $G(n,p)$ 
{\it conditioned\/} on the number of edges being equal $m$, and for $\bar p$
the number of edges in $G(n,p)$ is sharply concentrated around its expected value, i.e.\
$\bar p\binom{n}{2}$.

\begin{theorem} [Stepanov] \label{thm:Step}
 If $$p =  \frac{\log{n} + c}{n},$$
where $c \in \R$ is constant, then w.h.p.\
$G(n,p)$ consists of a giant component and isolated vertices, and the number of isolated vertices is asymptotic to Poisson with mean $e^{-c}$.  In particular
  $$\prob [ G(n,p) \mbox{ is connected}\, ] \to e^{-e^{-c}},$$ 
as $n \to \infty.$
\end{theorem}

Consequently, $\bar{p}=\log{n} / n$ is a sharp threshold for connectedness. 

\begin{theorem} \label{cor:Step}   Let $\omega \to \infty$ arbitrarily slowly. If 
$$p =  \frac{\log{n} + \omega}{n},$$
then w.h.p.\ $G(n,p)$ is connected, and if
$$p =  \frac{\log{n} - \omega}{n},$$
then w.h.p.\ $G(n,p)$ is disconnected.
\end{theorem}

Nowadays, Theorem  \ref{thm:ER} and Theorem  \ref{thm:Step} may be viewed as essentially equivalent,
thanks to  general ``transfer''  theorems, e.g.\ Janson, \L uczak and Ruci\'nski \cite{jlr},
Propositions 1.12, 1.13.

An important advantage of the  Erd\H{o}s--R\'enyi random graph $G(n,m)$ is that it can be
gainfully viewed as a snapshot of a natural  random graph process $ \{ G(n,M) \}$, $0\le M
\le \binom{n}{2}$, at ``time'' $M=m$. Here $G(n,M)$ is obtained from $G(n,M-1)$ by selecting the
location of $M$-th edge uniformly at random among all $\binom{n}{2}-(M-1)$ still available
options. As a special case of a result of Bollob\'as and Thomasson \cite{bt},
we have

\begin{theorem}\label{BT} For almost all realizations of the $\{G(n,M)\}$  process, 
$$
\min\{M:\text{min degree of } G(n,M)>0\}=\min\{M:G(n,M)\text{ is connected}\}.
$$
\end{theorem}

In a seminal paper \cite{lm}  Linial and Meshulam defined random $2$-dimensional simplicial complexes and found a two-dimensional cohomological analogue of Theorem \ref{cor:Step}.  
Subsequently Meshulam and Wallach \cite{mw} managed to extend the result of \cite{lm} to all
dimensions $k\ge 2$. These papers  have inspired  several other articles exploring the topology of random simplicial complexes, e.g.\ see Aronshtam et al \cite{al}, Babson et al.\ \cite{bh}, Bollob\'as and Riordan \cite{br}, and Kozlov \cite{ko}.

Our main goal in this article is to establish some $k$-dimensional analogues of Theorems \ref{thm:Step} and \ref{BT}, based on,  or inspired by,  the Linial--Meshulam and Meshulam--Wallach theorems.

\subsection{Topological preliminaries}

 In this subsection we define simplicial complexes and simplicial cohomology with $(\Z / 2)$-coefficients.  For a more complete introduction we refer the reader to the first two chapters of Hatcher's book \cite{Hatcher}.

An abstract {\it simplicial complex} is a finite set $V$, called the {\it vertices} of $S$ and a collection $S$ of subsets of $V$ such that
\begin{itemize}
\item $\{ v \} \in S$ for every $v \in V$,
\item  if $A \in S$ and $B \subset A$ is nonempty then $B \in S$.
\end{itemize}

Elements $\{ x,y\} \in S$ of cardinality $2$ are sometimes called {\it edges}, and elements $\{ x,y,z \}$ of cardinality $3$ {\it triangles}.  In general, elements of $S$ are called {\it faces}.

The {\it dimension} of a face $f \in S$ is $|f|-1$, where $|f|$ denotes the cardinality of $f$.  (So vertices are $0$-dimensional, edges are $1$-dimensional, etc.)  The dimension of  $S$ is the maximum cardinality of its faces.  Note that a simplicial complex of dimension $1$ corresponds to a simple graph  (i.e.\ a graph with no loops or multiple edges). 

One may also consider the {\it geometric realization} of $S$, sometimes denoted $|S|$, as a topological space.  We will abuse notation and identify $S$ with $|S|$.  In practice, it is clear whether one is talking about a combinatorial feature of $S$ or a topological feature.

Let $F_k=F_k(S)$ denote the set of $k$-dimensional faces of $S$.  Then the {\it $k$-cochains} $C^k$ is the vector space of functions $f: F_k \to \Z /2$.  There is a {\it coboundary map} $d_k: C^k \to C^{k+1}$, defined by
$$d_k f (\sigma) = \sum_{\tau} f(\tau),$$
where the sum is over all faces $\tau \subset \sigma$ such that $\dim \tau = \dim \sigma -1=k$.
If we introduce an $|F_k|\times |F_{k+1}|$ incidence matrix $I_k$ such that $I_k(\tau,\sigma)=$
if and only if $\tau\subset\sigma$, and view $f:=\{f(\tau)\}$, $g:=\{d_k f(\sigma)\}$ as vectors, then $g^T=f^T I_k$.

The $k$-cocycles is defined to be the subspace $Z^k = \ker d_k$, i.e.\  the left null-space
of $I_k$,  and the $k$-coboundaries is the subspace $B^k = \mbox{im} \, d_{k-1}$, i.e.\ 
the row space of $I_{k-1}$.  For each $f\in B^k$ there exists $A\subset F_{k-1}$ such that
$f$ is supported by the faces $\beta\in F_k(S)$ with a property:  $\beta$ has an odd number of $(k-1)$-faces
$\alpha\in F_{k-1}$. In particular, for $k=2$, $f$ is supported by the cut/set separating $A$ and
$A^c=[n]\setminus A$. 

It is easy to verify that $B^k \subseteq Z^k$, i.e.\ $d_k \circ d_{k-1} = 0$ for every $k$, or equivalently $I_{k-1}I_k=0$; indeed, given $\alpha\in F_{k-1},\,
\beta\in F_{k+1}$, $\alpha$ is either a face of exactly two $k$-faces of $\beta$, or not a face
of any $k$-face of $\beta$.

Then the $k$th cohomology is defined to be the quotient vector space $H^k = Z^k / B^k$.  (We might write $H^k(S, \Z/2)$ to emphasize that this is the cohomology for the simplicial complex $S$, and that we mean cohomology with $\Z / 2$ coefficients.) We are especially interested in the case $H^k =  0$, which means that every $k$-cocycle is a $k$-coboundary.

If we view $\emptyset$ as $(-1)$-dimensional face, this is sometimes called {\it reduced} cohomology, and denoted by $\widetilde{H}^k$ rather than $H^k$. Note that $\widetilde{H}^k = H^k$ except in the case $k=0$.  In the reduced cohomology case $I_0$ is a single row with all entries $1$, while $I_1$ is a vertex-edge incidence matrix of a simple graph $G$ on $[n]$.
Consequently $\widetilde{H}^0=0$ if and only if $G$ is connected. In general, the dimension of $\widetilde{H}^0$ is
$c(G)-1$, where $c(G)$ is the number of connected components of $G$.

A topological aside:  One may just as easily talk about {\it homology} $H_k$ rather than cohomology $H^k$, but in the cases we are interested in the results would be exactly the same.  (It is pointed out in \cite{lm} that this equivalence follows from universal coefficients.) However the main argument seems to be easier to make in terms of cohomology than in terms of homology.

\subsection{The Linial--Meshulam and Meshulam--Wallach theorems}

The random $k$-dimensional simplicial complex $Y \sim Y(n,p)$ has vertices $[n] = \{ 1, 2, \dots, n\}$, and complete $(k-1)$-skeleton, meaning that $Y$ contains all subsets of $[n]$ of
cardinality $k$. Then each $k$-face (subset of cardinality $k+1$)  is included in $Y$ with probability $p$, independently of all other such faces.  The following is a cohomological analogue of Theorem \ref{thm:ER},
$k=2$ and $k\ge 2$ versions being proved by Linial and Meshulam \cite{lm} and by Meshulam and Wallach 
\cite{mw} respectively.

\begin{theorem}\label{thm:LM}  Let $\omega \to \infty$ arbitrarily slowly and $Y \sim Y(n,p)$.  
\begin{enumerate}
\item If $$p \ge  \frac{k\log{n} + \omega}{n},$$
then  $$\prob [ H^{k-1}(Y, \Z /2) = 0] \to 1;$$ 
\item if $$p \le  \frac{k\log{n} - \omega}{n},$$ then
  $$\prob [ H^{k-1}(Y, \Z/2) = 0] \to 0,$$ as $n \to \infty$.
\end{enumerate}
\end{theorem}

(In \cite{mw} the same statement is shown to hold, even when $\Z / 2$ is replaced by any finite abelian group of coefficients.)

Part (i) is the heart of Theorem \ref{thm:LM}, as part (ii) is relatively straightforward. Indeed, for $p\le (k \log n -\omega)/n$,
w.h.p.\ at least one $\alpha\in F_{k-1}$ is not a face of any $k$-face in $Y(n,p)$, i.e.\ $\alpha$
is {\it isolated\/}.
The characteristic function $f$ of such an $\alpha$ then is a cocycle, by default,  but it is not a coboundary; indeed, for  the $k$-faces $\sigma\notin F_k(Y(n,p))$,  such that
$\alpha\subset\sigma$,
$$
\sum_{\tau\subset \sigma}f(\tau)=1\neq 0.
$$
On the other hand, one sees that, for $p\ge n^{-1}(k\log n+\omega)$, 
w.h.p.\ there are no isolated $(k-1)$-faces. 
This suggests that isolated $(k-1)$-faces might hopefully be the {\it most\/}  likely obstruction to cohomological connectedness of $Y(n,p)$ if $p\ge n^{-1}(k\log n+w)$. That was exactly the motivation behind the statement and the proof of the key part (i) in \cite{lm}, \cite{mw}. In fact, 
we shall see that, in a closer analogy with $G(n,p)$, isolated $(k-1)$-faces are  the {\it only\/} likely obstruction to such connectedness even ``earlier'', when $p=n^{-1}(k\log n +O(1))$.

\subsection{Notions of connectivity} 

Linial and Meshulam introduced the terminology ``(co)homological connectedness'' to emphasize that Theorem \ref{thm:LM} should be viewed as a $2$-dimensional analogue of the Erd\H{o}s--R\'enyi theorem.  Spaces where every (co)cycle is a (co)boundary are also sometimes called ``acyclic'', or to have ``vanishing (co)homology.''

We call a $k$-dimensional simplicial complex $S$  {\it hypergraph connected} if for every two
$(k-1)$-faces $\alpha$, $\alpha^\prime \in F_{k-1}(S)$, there exists a sequence of $(k-1)$-faces
$$\alpha=\alpha_1,\alpha_2,\dots,\alpha_k=\alpha^\prime$$ that joins $\alpha$ and 
$\alpha^\prime$, in a sense that for $i$, $\alpha_i\cup
\alpha_{i+1}\in F_k(S)$.

\begin{theorem} \label{thm:conn} Let $S$ be a $k$-dimensional complex with complete
$(k-1)$-skeleton (i.e.\ $|F_{k-1}(S)| = {n \choose k}$). If $H^{k-1}(S,\Z/2)=0$, then $S$ is hypergraph connected.
\end{theorem}

\begin{proof}[Proof of Theorem \ref{thm:conn}] We use induction on $k$. The statement obviously holds for $k=1$. Suppose it is true for some $k\ge 1$. Let $S$ be a $(k+1)$-dimensional
complex such that $H^k(S,\Z/2)=0$. Define the link
of a vertex $v\in [n]$ by
$$
\link_{S}(v) = \{ \sigma - v : \sigma\in S,\,v \in \sigma \}.$$
Note that $\link_{S}(v)$ is itself a simplicial complex, and 
$$
\dim(\link_{S}(v)) \le \dim( S) - 1.
$$ 
Then $H^{k-1}(\link_{S}(v))=0$ for all $v\in [n]$. If not, there is  $v^*$ such that
$H^{k-1}(\link_{S}(v^*))$ $\neq 0$. Thus there exists a cocycle $g:F_{k-1}(\link_{S}(v^*))\to\{0,1\}$ which is not a coboundary, the latter meaning that for some $k$-face $\sigma$, that does not contain $v^*$,
$$
\sum_{\tau\in \sigma}g(\tau)=1.
$$
Define $f:F_k(S)\to\{0,1\}$ by the conditions
\si
(a)  $f((v^*,\tau))=g(\tau)$ for $\tau\in F_{k-1}(\link_{S}(v^*))$; 
\si
(b) $f(\alpha)=0$ for all $k$-faces $\alpha\neq (v^*,\tau)$, with $\tau\in F_{k-1}(\link_{S}(v^*))$. 
\si
Then $f$ is a cocycle of $S$, but $f$ is not a coboundary, because
$$
\sum_{\alpha\in (v^*,\sigma)}f(\alpha)=\sum_{\tau\in \sigma}g(\tau)=1.
$$
Contradiction! So indeed $H^{k-1}(\link_{S}(v))=0$ for all $v\in [n]$. 
By induction hypothesis, each $\link_S(v)$ is hypergraph connected, so that in $\link_S(v)$ every two $(k-1)$-faces are joined by a path of $(k-1)$-faces.

It remains to show that $S$ itself is hypergraph connected. Let $\alpha,\,\alpha^\prime\in F_k(S)$.
Define $t=|\alpha\cap\alpha^\prime|$. If $t>0$, then $\alpha =(v_1,v_2,\dots,v_k)$,
$\alpha^\prime=(v_1,v_2^\prime,\dots,v_k)$. Since $\link_S(v_1)$ is hypergraph connected, $(v_2,\dots,v_k)$ and $(v_2^\prime,\dots,v_k^\prime)$ are joined by a path of
$(k-1)$-faces in $\link_S(v_1)$. Augmenting these intermediate faces with $v_1$ we obtain
a path joining $\alpha$ and $\alpha^\prime$ in $S$. Suppose $t=0$, so that
$(v_1,\dots,v_k)$ and $(v_1^\prime,\dots,v_k^\prime)$ do not overlap. By the preceding
argument, $(v_1,v_2,\dots,v_k)$ and $(v_1,v_2^\prime,\dots,v_k^\prime)$ are joined by a path in $S$,
and so are  $(v_1,v_2^\prime,\dots,v_k^\prime)$ and $(v_1^\prime,v_2^\prime,\dots,v_k^\prime)$.
Concatenating these two paths we get a path from $(v_1,\dots,v_k)$ to $(v_1^\prime,\dots,
v_k^\prime)$ in $S$. 
\end{proof}

In light of Theorem \ref{thm:conn}, Theorem \ref{thm:LM} effectively shows that $\bar{p}= n^{-1}k 
\log n$ is both the threshold for $H_{k-1}(Y, \Z/2) = 0$ and  the threshold for the hypergraph connectedness of the underlying hypergraph.

\subsection{Main results} Our first result is a hypergraph analogue of Theorem
\ref{thm:Step}. Let $k\ge 1$. Let $HG(n,p)$ denote the random hypergraph induced by the random
complex $Y(n,p)$. The hypervertex set and the hyperedge set of $HG(n,p)$ are
$$F_{k-1}=\binom{[n]}{[k]}$$ and $$F_k(Y(n,p))\subseteq F_k=\binom{[n]}{[k+1]},$$ respectively.
\begin{theorem}\label{thm:conn,k} If
$$
p=\frac{k\log n +c}{n},
$$
where $c\in \Bbb R$ is constant, then w.h.p.\ $HG(n,p)$ consists of a giant component
and isolated vertices, and the number of those is asymptotically Poisson, with mean
$e^{-c}/k!$, and hence the probability of hypergraph connectedness approaches $\exp( e^{-c} / k!)$ as $n \to \infty$.
Consequently, $p=n^{-1} k \log n$ is a sharp threshold probability for connectedness property of
$HG(n,p)$.
\end{theorem}

A key estimate in an unexpectedly simple proof is obtained by using links and induction on $k$.\\

Analogously to $\{G(n,M)\}$, let us introduce the random complex process \linebreak $\{Y(n,M)\}$,
where $Y(n,M)$ is a uniformly random $k$-dimensional complex with $M$ $k$-faces. Each $Y(n,M)$ is obtained from $Y(n,M-1)$ by choosing the location of $M$-th $k$-face uniformly at random
among all $\binom{n}{k+1}-(M-1)$ possibilities. Here is a $k$-dimensional extension of
Bollob\'as-Thomasson's Theorem \ref{BT}.
\begin{theorem}\label{BT,k} For almost all realizations of the $k$-dimensional process $\{Y(n,M)\}$, the
random $k$-face, that eliminates the chronologically last isolated $(k-1)$-face, also
makes the resulting complex hypergraph connected.
\end{theorem}
\bi
 
Furthermore we sharpen the Linial-Meshulam and Meshulam-Wallach theorems.
%For $k=2$ we establish the cohomological counterparts of Theorem \ref{thm:conn,k} and Theorem \ref{BT,k}.
\begin{theorem}\label{thm:homconn,2}  If
$$
p=\frac{k\log n +c}{n},
$$
where $c\in \Bbb R$ is constant, then $\beta^{k-1} := \dim(H^{k-1}(Y, \Z/2))$ is asymptotically Poisson, with
mean $e^{-c}/k!$. In particular, $H^{k-1}(Y, \Z/2)$ vanishes with limiting probability
$\exp (-e^{-c} / k!)$. 
\end{theorem}
\si
Since $e^{-e^{-c}/k!}\to 0$ if $c\to-\infty$, and $e^{-e^{-c}/k!}\to 1$ if $c\to\infty$, Theorem 
\ref{thm:homconn,2} implies Theorem \ref{thm:LM}.
We should also note that Theorem \ref{thm:homconn,2} in combination with Theorem \ref{thm:conn}, imply Theorem \ref{thm:Step} as a direct byproduct.\\
 
For $k=2$ we prove a cohomological extension of Theorem \ref{BT}.
\begin{theorem}\label{thm:homBT,2} For almost all realizations of the $2$-dimensional process 
$\{Y(n,M)\}$, the $2$-face (triangle), that eliminates the chronologically last isolated $1$-face (edge), also makes $H^1(Y, \Z/2) = 0$.
\end{theorem}
Thus,  while the three random moments,
\begin{align*}
M_1=:&\,\min\{m: Y(n,m)\text{ has no isolated edge}\},\\
M_2=:&\,\min\{m: Y(n,m)\text{ is hypergraph connected}\},\\
M_3=:&\,\min\{m: H^1(Y(n,m))\text{ vanishes}\},
\end{align*}
(always obeying $M_1\leq M_2\leq M_3$), may generally be distinct, 
the event $\{M_1=M_2=M_3\}$ has probability approaching $1$ as $n\to\infty$.\\

A key part of our proof is based on counting non-trivial cocycles by the degree sequences of their 
supports, an approach considerably simpler than deep counting arguments in \cite{lm},
\cite{mw}.  We conjecture that the extension of Theorem \ref{thm:homBT,2}
holds for all $k\ge 2$, and for cohomology with coefficients in any finite abelian group.
Since the proofs of Theorems \ref{thm:conn,k} and \ref{BT,k} are relatively simple, we wonder whether such an extension could be proved by using links and induction on $k$. It may well
be possible also to get it done by a proper modification of the method  in \cite{mw}, but
we haven't explored this route.

\section{Proofs of Theorem \ref{thm:conn,k} and Theorem \ref{BT,k}.} 

{\bf (1)\/} Let $A_n$ be the event that all non-isolated $(k-1)$-faces of $Y \sim Y_k(n,p)$  belong to the same component, or equivalently that every two non-isolated
$(k-1)$-faces are joined by a path in $Y$, in the sense of ``hypergraph connected'' described above. 

Given a vertex $v\in [n]$, define the vertex link
$$
\link_Y(v)=\{\sigma-v : \sigma\in Y(n,p),\,v\in \sigma]\}.
$$
So each $\link_Y(v)$ is a $(k-1)$-dimensional complex, distributed as $Y_{k-1}(n-1,p)$.
Of course, $\link_Y(v)$, $v\in [n]$, are interdependent.  Let $ A_n(v)$ be the
the event that every two non-isolated $(k-2)$-faces of $\link_Y(v)$ are joined
by a path in $\link_Y(v)$. Let  $B_n$ be the event that for every $\alpha,\,\alpha^\prime\in 
F_{k-1}(Y(n;p))$,
with $\alpha\cap\alpha^\prime =\emptyset$, there exist $v\in \alpha$ and $v^\prime\in \alpha^\prime$ such that
$(\alpha^\prime\setminus\{v^\prime\})\cup\{v\}$ is a non-isolated $(k-1)$-face of $Y(n,p)$. Then
\begin{equation}\label{Ansupset}
A_n\supseteq \left(\bigcap_{v\in [n]} A_n(v)\right)\bigcap \, B_n.
\end{equation}
Indeed, suppose that the RHS event in \eqref{Ansupset} holds. Let $\alpha,\alpha^\prime\in F_{k-1}(Y(n,p))$ be non-isolated. If there is
$v\in \alpha\cap \alpha^\prime$ then $\alpha\setminus \{v\}$ and $\alpha^\prime\setminus
\{v\}$ are non-isolated $(k-2)$-faces in $\link_Y(v)$, whence they are joined by a path
in $\link_Y(v)$. By the definition of $\link_Y(v)$, augmenting the edges of this path with
$v$, we get a path joining $\alpha$ and $\alpha^\prime$ in $Y(n,p)$. Suppose that 
$\alpha\cap\alpha^\prime=\emptyset$. Then, by the definition of $B_n$,
there exist $v\in \alpha$ and $v^\prime\in \alpha^\prime$ such that
$\alpha^{\prime\prime}:=(\alpha^\prime\setminus\{v^\prime\})\cup\{v\}$ is non-isolated. 
By the first part, $\alpha$ and $\alpha^{\prime\prime}$ are joined by a path in $Y(n,p)$,
and likewise so are $\alpha^{\prime\prime}$ and $\alpha^\prime$.
\si

Let $g_k(n;p)=\pr(A_n^c)$, and $h_k(n;p)=\pr(B_n^c)$. 
Then \eqref{Ansupset} implies a recurrence inequality
\begin{equation}\label{ineqforg}
g_t(n;p)\le ng_{t-1}(n-1;p) +h_t(n;p),\quad t\ge 2.
\end{equation}
Let us  bound $h_t(n;p)$. We observe that the number of of ordered pairs of disjoint
$\alpha,\,\alpha^\prime\in F_{t-1}(Y(n,p)$  is less than $\binom{n}{t}^2$, and the number of pairs $(v,v^\prime)$, 
$v\in \alpha$, $v^\prime\in \alpha$, is $t^2$. The probability that for every such pair $(v,v^\prime)$
there does not exist $u\in [n] \setminus (\alpha\cup \alpha^\prime)$ such 
$(\alpha^\prime\setminus \{v^\prime\})\cup\{v\}\cup\{u\}$ is in $F_t(Y(n,p))$
is $q^{(n-2t)t^2}$.
Therefore, for $t\le k$,
$$
h_t(n;p)\le \binom{n}{t}^2q^{(n-2t)t^2}\le a n^{2t}q^{nt^2},\quad a:=\frac{e^{2k^2}}{k!}.
$$
So \eqref{ineqforg} simplifies to
\begin{equation}\label{simplineqforg}
g_t(n;p)\le ng_{t-1}(n-1;p)+ an^{2t}q^{nt^2}.
\end{equation}
Since $q^n=O(n^{-k})$, an easy induction shows that 
\begin{equation}\label{ind}
g_t(n;p)\le n^{t-1}g_1(n-t+1;p)+2an^{t+2}q^{4n},\quad 2\le t\le k.
\end{equation}
Consider \eqref{ind} for $t=k$. If 
\begin{equation}\label{|c(n)|small}
p=\frac{k\log n +c(n)}{n},\quad |c(n)|=o(\log n),
\end{equation}
then
$$
n^{k+2}q^{4k}=O(n^{k+2}n^{-4k}e^{4|c(n)|})=O(n^{-2k+2}e^{4|c(n)|})\to 0,
$$
as $k\ge 2$. Furthermore, $g_1(n-k+1;p)$ is the probability that $G(n-k+1;p)$ has a
component of size from $2$ to $(n-k+1)/2$, which---for $p$ in question---is bounded by
twice the expected number of components of size $2$. And this expected value is of order
$$
n^2p(1-p)^{2n}\le n^2pe^{-2np}=O\left(\frac{e^{2|c(n)|}\log n}{n^{2k-1}}\right).
$$
So the first term in the RHS of \eqref{ind} is of order $O(n^{-k}e^{2|c(n)|}\log n)$. In summary,
$$
\Bbb P(A_n^c)=O(n^{-k}e^{2|c(n)|}\log n).
$$
Thus, under condition \eqref{|c(n)|small},
\begin{equation}\label{PAn>}
\Bbb P(A_n)=1-O(n^{-k}e^{2|c(n)|}\log n)\to 1.
\end{equation}
\si

It remains to show that, for $p=(k\log n+c)/n$, $X_n$ the total number of isolated $(k-1)$-faces of
$Y(n,p)$ is asymptotically Poisson, with mean $e^{-c}/k!$. This is done by
a standard argument based on factorial moments. 
So $Y(n,p)$ is connected
with the limiting probability $e^{-e^{-c}/k!}$. The proof of Theorem \ref{thm:Step} is
complete. \qed
\si
{\bf Note.\/} If $c(n)=-\log \log n$, say, then the second order moment method shows that  
\begin{equation}\label{Xnlarge}
\frac{X_n}{(\log n)^{1/k!}}\to 1,\quad\text{in probability}.
\end{equation}
\bi

{\bf (2)\/} Let us prove Theorem \ref{BT,k}. We embed the random complex process
$\{Y(n,M)\}$ into a continuous-time Markov process $\{Y_t(n)\}$, $t\ge 0$. To do so, 
we introduce  i.i.d. random variables $T_{\sigma}$, $\sigma\in F_k$, with $\Bbb P(T_{\sigma}\le t)=1-e^{-t}$. $T_{\sigma}$ can be interpreted as a waiting time till  ``birth'' of the $k$-face $\sigma$. We define 
$$
Y_t(n)=\{\sigma\in F_k: T_{\sigma}\le t\}.
$$
So $Y_t(n)$ is a complex whose $k$-faces have been born up to time $t$; $Y_0(n)$ is the
complete $(k-1)$-dimensional complex, and $Y_{\infty}(n)$ is the complete $k$-dimensional complex. Clearly, 
$Y_t(n)$ is a Bernoulli complex $Y(n,p)$ with $p=p(t)$. Also, $\{Y_t(n)\}_{t\ge 0}$
is a Markov process, thanks to memoryless property of the exponential distribution. Introduce a sequence $\{t(M)\}$ of stopping times such
$$
t(M)=\min\{t\ge 0 : |\{\sigma:T_{\sigma}\le t\}|=M\};
$$
in words, $t(M)$ is the first time $t$ the number of $k$-faces reaches $M$. Then (1) each 
$Y_{t(M)}(n)$ is distributed uniformly on the set of all complexes with $M$ $k$-faces, 
and (2) conditioned on $Y_{t(M-1)}(n)$, the location of $M$-th $k$-face in $Y_{t(M)}(n)$ is
distributed uniformly on the set of all $\binom{n}{k+1}-(M-1)$ available locations. Thus
$\{Y_{t(M)}(n)\}$ is distributed as $\{Y(n,M)\}$. \si

Introduce 
$$
p_1=\frac{k\log n -\log\log n}{n},\quad p_2=\frac{k\log n +\log\log n}{n},
$$
and $t_i$ defined by $p_i=1-e^{-t_i}$. Since $Y_{t}(n)$ is distributed as $Y(n,p(t))$, it follows
from \eqref{Xnlarge} that w.h.p.\ $Y_{t_1}(n)$ consists of $X_n\sim(\log n)^{1/k!}$ isolated
$(k-1)$-faces and a single component on the remaining $\left[\binom{n}{k}-X_n\right]$ $(k-1)$-faces.  As for $Y_{t_2}(n)$, w.h.p.\ it consists of a single component.  Let 
$\tau$ be the first time $t$ when the number of isolated $(k-1)$-faces drops down by two
or more. Then
\begin{align*}
\Bbb P(\tau\le t_2\mid Y_{t_1}(n))\le&\, X_n^2\int_{t_1}^{t_2}e^{-t}\,dt\\
\le&\,X_n^2(t_2-t_1)=O\bigl(n^{-1}(\log n)^{2/k!}\log\log n\bigr)\to 0.
\end{align*}
Here $X_n^2$ is a crude upper bound for the number of pairs of 
$(k-1)$-faces, isolated at time $t_1$, that happen to be the faces of the same $k$-simplex. And $e^{-t}$ is the 
probability density of the birth time for such a $k$-simplex. 

So, w.h.p.\ throughout $[t_1,t_2]$ the complex $Y_t(n)$ continues to be a giant component
plus a set of isolated $(k-1)$-faces, gradually swallowed, one such face at a time, by the
current giant component. Thus, w.h.p.\ $Y_t(n)$ becomes connected when the
last isolated $(k-1)$-face gets joined by a newly born $k$-simplex to the current giant
component.  Consequently, the same property holds for
the subprocess $\{Y_{t(M)}(n)\}\overset{\Cal D}\equiv \{Y(n,M)\}$.\qed

\section{Proof of Theorem \ref{thm:homconn,2} for $k=2$} 

The reason we present an argument  for $k=2$ separately is that our proof of Theorem \ref{thm:homBT,2}  is essentially  this argument's  follow-up.  

Thus we consider $Y(n,p):=Y_2(n,p)$, the Bernoulli $2$-dimensional complex with the complete 
$1$-dimensional skeleton. Our main task is to bound the expected number 
of non-trivial cocycles for $p$ close to $(2\log n)/n$.  
\si

A $1$-cocycle $f$ induces a graph $G=G(f)$ on the vertex set $[n]$ with the edge set $E(G)=
\{\bold u\in F_{1}\,:\,f(\bold u)=1\}$, i.e.\  the support of $f$.  Let
$\bold d =\bold d_G=\{d_G(\bold v)\}=\{d(\bold v)\}$ be the degree sequence of $G=G(f)$.
A key idea of \cite{lm}, \cite{mw} was to focus on non-trivial cocycles $f$
with the smallest $|E(G(f))|$. These extremal cocycles have three crucial properties. 
\si

First of all, 
it turned out that, for every such cocycle $f$, 
\begin{equation}\label{maxdv<}
\max_{\bold v\in [n]} d\bigl(\bold v(G(f))\bigr)\le D:=\left\lfloor \frac{n-1}{2}\right\rfloor,
\end{equation}
a crucial improvement of the trivial bound $n-1$.  (For $k\ge 2$, the bound is $\lfloor (n-k+1)/2\rfloor$.)
\si

Second, the graph $G(f)$ has a single non-trivial component.
\si

To formulate the third, rather subtle, property, introduce $X(f)$, the number of such triangles that contain an odd number,
$1$ or $3$, edges from $E(G(f))$. Then
\begin{equation}\label{X(f)>nm/3}
 X(f)\ge \frac{n |E(G(f))|}{3}.
 \end{equation}
(For the $k$-dimensional complex, the lower bound is $n|E(f)|/(k+1)$, \cite{mw}.)

Why does $X(f)$ matter so much? Because the probability $P_n$ that $Y(n,p)$ has a non-trivial
$1$-cocycle is {\it at most\/}  the expected number of $1$-{\it cochains\/} $f$,
having those three properties, such that $Y(n,p)$ does not contain any one of $X(f)$ triangles.
The probability of this event is
$$
(1-p)^{X(G(f))}\le e^{-pX(G(f))}.
$$
Therefore
\begin{equation}\label{Pn<s(m)}
P_n\le\,\sum_{m\ge 1}s(m),\quad s(m):=\sum_{G: e(G)=m}e^{-pX(G)};
\end{equation}
here $X(G)$ is the total number of triangles that contain an odd number of edges of $G$,
and the sum is over all graphs $G$ with $e(G)=m$ edges, of maxdegree $\le\lfloor (n-1)/2\rfloor$,
with a single non-trivial component, and $X(G)\ge ne(G)/3$.
\begin{Lemma}\label{p=(2logn+xn)/n} Let
\begin{equation}\label{p>}
p=\frac{2\log n +x_n}{n},\quad |x_n|=o(\log n).
\end{equation}
Then (1)
\begin{equation}\label{EG>1}
\sum_{G\,:\, e(G)>1}\exp\bigl[-pX(G)\bigr]\le_b n^{-1}e^{2|x_n|}\to 0,
\end{equation}
and (2)
\begin{equation}\label{eG=1}
\sum_{G\,:\, e(G)=1}\exp\bigl[-pX(G)\bigr]\sim \frac{e^{-x_n}}{2}.
\end{equation}
\end{Lemma}
\si
{\bf Proof of Lemma \ref{p=(2logn+xn)/n}.\/} The cases $e(G)=O(n)$ and $e(G)>n^{1+\varepsilon}$
are relatively simple, and it is the intermediate values of $e(G)$ where our argument truly
differs from those in \cite{lm}, \cite{mw}. To be sure, our treatment of $e(G)=O(n)$ is different enough 
to cover $p=(2 \log n +x_n)$, $|x_n|=o(\log n)$, compared with $x_n\to\infty$ in \cite{lm}-\cite{mw}.\\

Let $\nu=\nu(G)$ denote the number of vertices, and $m=e(G)$ the number of edges 
in a non-trivial component $C=C(G)$ of a generic graph $G$ in question. Then $m\ge\nu-1$.
\si

{\bf (1)\/} Let $\nu\le an$, where $a\in (0,1/4)$. Obviously
\begin{equation}\label{triv}
X(G)\ge (n-\nu)m.
\end{equation}
The total number of graphs $G$ on $[n]$ with $m$ edges and $|V(C(G))|=\nu$ is at most
$$
\binom{n}{\nu}\binom{\binom{\nu}{2}}{m}\le \binom{n}{\nu}\left(\frac{e\binom{\nu}{2}}{m}\right)^{m}.
$$
So
$$
\sum_{G:\, |V(C)|=\nu,\,|E(C)|=m}\exp\bigl[-pX(G)\bigr]\le\, S(\nu,m),
$$
where
$$
S(\nu,m):=\binom{n}{\nu} \left(\frac{e\binom{\nu}{2}}{m}\right)^{m}\exp\bigl[-p(n-\nu)m\bigr].
$$
Now, using $m+1\ge\nu$,
\begin{align*}
\frac{S(\nu,m+1)}{S(\nu,m)}\le&\, \frac{\nu^2}{m+1}\exp\bigl[-p(n-\nu)\bigr]\\
\le&\, \nu\exp\bigl[-p(n-\nu)\bigr]\le \frac{e^{|x_n|}}{n^{1-2a}}\to 0,
\end{align*}
as $a<1/2$. Hence, for $\nu\le an$, 
\begin{align*}
\sum_{m\ge\nu-1}S(\nu,m)\le& 2S(\nu,\nu-1)\le_b S(\nu);\\
S(\nu):=&\,(4n)^{\nu}\exp\bigl[-p(n-\nu)(\nu-1)\bigr].
\end{align*}
Now, for $\nu\le an$,
$$
\frac{S(\nu+1)}{S(\nu)}=4n\exp\bigl[-p(n-2\nu)\bigr]\le 4\frac{e^{|x_n|}}{n^{1-4a}}\to 0,
$$
as $a<1/4$. Therefore
\begin{equation}\label{3lenulean}
\sum_{G\,:\,3\le |V(C)|\le an}\exp\bigl[-pX(G)\bigr]=O\bigl( n^{-1}e^{2|x_n|}\bigr).
\end{equation}
And, of course,
\begin{equation}\label{nu=2}
\sum_{G\,:\,|V(C)|=2,\,|E(C)|=1}\exp\bigl[-pX(G)\bigr]=\binom{n}{2}\exp\bigl[-p(n-2)\bigr]\sim \frac{1}{2}e^{-x_n}.
\end{equation}
\si

{\bf (2)\/} Let $m\ge m_n:=3n^{4/3}e^{|x_n|}$. Using $X(G)\ge nm/3$, we have
\begin{align*}
s(m)\le&\, \sum_{\nu}\binom{n}{\nu}\binom{\binom{\nu}{2}}
{m}\cdot\exp\bigl[-pnm/3\bigr] \\
\le&\, 2^n\left(\frac{2n^2}{m}\right)^{m}\cdot\exp\bigl[-pnm/3\bigr] =
2^n \left(\frac{2n^2}{m e^{pn/3}}\right)^{m}\\
\le& 2^n\left(\frac{2n^{4/3}e^{|x_n|}}{m}\right)^{m}.
\end{align*}
Consequently
\begin{equation}\label{mularge}
\sum_{m\ge  m_n}\!\!\!s(m)\le \exp\bigl(n\log 2-n^{4/3}\log 2\bigr)\to 0,
\end{equation}
superexponentially fast.
\si

{\bf (3)\/} It remains to consider $\nu\ge an$ and $m\le m_n$. It is crucial that 
$m/\nu^2\to 0$ in this range. By symmetry,
\begin{equation}\label{bysymm}
\sum_{G: e(G)=m,|V(C(G))|=\nu}\exp\bigl[-pX(G)\bigr]=\binom{n}{\nu}
\sum_{G:e(G)=m,V(C(G))=[\nu]}\exp\bigl[-pX(G)\bigr].
\end{equation}
We need to show that the RHS tends to $0$ as $n\to\infty$. 

Let 
$\bold d=(d_1,\dots,d_{\nu})$ be the generic vertex degrees of $C(G)$;
so
\begin{equation}\label{|d|=}
\|\bold d\|:=\sum_u d_u=2\mu,\quad 1\le d_u\le n/2,\,\,\,\forall\, u\in [\nu].
\end{equation}
Notice upfront that the total number of such graphs, connected or not,  is bounded above by
\begin{equation}\label{graphs<d}
(2\mu-1)!!\prod_u \frac{1}{d_u!},
\end{equation}
see Bender and Canfield \cite{bc}. 

Our next step, logically, is to find a lower bound for
$X(G)$ in terms of $\bold d$. To this end, we first write 
\begin{equation}\label{XG=}
X(G)=(n-\nu)\mu+ Y_1(C)+Y_3(C),
\end{equation}
where $Y_1(C)$ ($Y_3(C)$ resp.) is the total number of triples $(u,v,w)$ from the $\nu$ vertices such that 
$(u,v)$ is an edge, and  $(u,w)$ and $(v,w)$ are not edges ($(u,w)$, $(v,w)$ are edges, resp.).  For a given edge $(u,v)$ the number of
$w\neq u,v$ such that at least one of $(u,w)$,  $(v,w)$ is  an edge is
\begin{align*}
&(d_u-1)+ (d_v-1)\\
- &|\{w\neq u,v\,:\, (u,w),\,(v,w)\text{ edges, both}\}|\\
=&\,(\nu-2)-|\{w\neq u,v\,:\, (u,w),\,(v,w)\text{ not edges}\}|,
\end{align*}
whence
\begin{multline*}
 |\{w\neq u,v\,:\, (u,w),\,(v,w)\text{ edges, both}\}|=(d_u+d_v)-\nu\\
 +|\{w\neq u,v\,:\, (u,w),\,(v,w)\text{ not edges}\}|.
 \end{multline*}
So, summing over all edges $(u,v)$, and noticing that every triangle in $C$
will be counted thrice,
\begin{align*}
3Y_3(C)=&\, \sum_{(u,v)\text{ edge}}(d_u+d_v)-\nu m+Y_1(C)\\
=&\,\frac{1}{2}\sum_{\{u,v\}:\,(u,v)\text{ edge}}(d_u+d_v)-\nu m +Y_1(C)\\
=&\,\sum_u d_u\sum_{v\neq u}\bold 1_{\{(u,v)\text{ edge}\}}-\nu m+Y_1(C)\\
=&\,\sum_u d_u^2-\nu m+Y_1(C).
\end{align*}
Consequently
\begin{equation}\label{Y13>}
Y_1(C)\ge  \nu m -\sum_u d_u^2,\quad Y_3(C)\ge\frac{1}{3}\left(\sum_ud_u^2-\nu m
\right).
\end{equation}
The second inequality is known,  see Lov\'asz \cite{lo}, Solution of Exercise 10.33.
Since
$$
\sum_u d_u^2 \ge \nu\left(\frac{2m}{\nu}\right)^2=\frac{4m^2}{\nu},
$$
it follows from \eqref{Y13>} that
\begin{equation}\label{Turan}
Y_3(C)\ge \frac{1}{3}\left(\frac{4m^2}{\nu}-\nu m\right),
\end{equation}
a classic inequality, due to Tu\' ran, useful  when $m> \nu^2/4$. However, in our case $m=o(\nu^2)$, so we pin our hopes on the lower bound for $Y_1(C)$ in \eqref{Y13>}. It gives
\begin{equation}\label{exp<exp(sumd2)}
\exp[-pX(G)]\le \exp\left(-pn m+p\sum_u d_u^2\right).
\end{equation}
So \eqref{bysymm}, \eqref{graphs<d} and \eqref{exp<exp(sumd2)} yield
\begin{equation}\label{sumnumu,degrees}
\begin{aligned}
\sum_{G:e(G)=m, \, |V(C(G))|=\nu}&\exp\bigl[-pX(G)\bigr]\\
\le&\, e^{-pnm}\binom{n}{\nu}(2m-1)!!\sum_{\bold d\text{ meets }\eqref{|d|=}}
\prod_u\frac{\exp(p\,d_u^2)}{d_u!}.
 \end{aligned}
 \end{equation}
Using $d!\ge (d/e)^d$, 
$$
\prod_u\frac{\exp(p\,d_u^2)}{d_u!}\le e^{H(\bold d)}, 
$$
where
$$
H(\bold d)=\sum_u \phi(d_u),\quad \phi(d)=d\log\frac{e}{d}+pd^2.
$$

Let us show that $\sum_{ v}\phi(d_{ v})$ is  ``negligible'', uniformly for $\{d_v\}$ in
question.
To this end, notice that
$$
\phi^{(2)}(d)=-\frac{1}{d}+2p
$$
is negative (positive resp.) for $d< \bar d$ ($d >\bar d$ resp.), where
\begin{equation}\label{bard=}
\bar d=\frac{1}{2p}=\frac{n}{2(2\log n+x_n)}=\Theta\left(\frac{n}{\log n}\right).
\end{equation}
That is, $\phi(d)$ is strictly concave  if $d\ge \bar d$. Since $d\in [\bar d,n/2]$ is a convex
combination of $\bar d$ and $D$,
$$
d=\frac{D-d}{D-\bar d}\,\bar d+\frac{d-\bar d}{D-\bar d}\,D,
$$
we have then
$$
\phi(d)\le  \frac{D-d}{D-\bar d\,}\,\phi(\bar d\,)+\frac{d-\bar d}{D-\bar d}\,\phi(D).
$$
So, introducing
\begin{equation}\label{defnu1,mu1}
\nu_1=\nu_1(\bold d):=|\{ v\,:\,d_{v}\ge \bar d\,\}|,\qquad
\mu_1=\mu_1(\bold d)=\sum_{ v\,:\,d_{ v}\ge \bar d} d_{ v},
\end{equation}
we have
\begin{multline}\label{sumphid>bard}
\sum_{\bold v\,:\,d_{\bold v}\ge \bar d}\phi(d)\le\frac{\phi(\bar d\,) }{D-\bar d}\,(\nu_1D-\mu_1)
+\frac{\phi(D)}{D-\bar d\,}\,(\mu_1-\nu_1 \bar d\,)\\
=\,-\nu_1\,\frac{\phi(D)\bar d-\phi(\bar d\,)D}{D-\bar d}+\mu_1\,\frac{\phi(D)-\phi(\bar d\,)}
{D-\bar d}.
\end{multline}
Direct computation shows that
\begin{align*}
\alpha_n:=&\,\frac{\phi(D)\bar d-\phi(\bar d)D}{D-\bar d}=\frac{n}{4}\bigl[
1+O\bigl((\log\log n +|x_n|)/\log n\bigr)\bigl],\\
\beta^n:=&\,\frac{\phi(D)-\phi(\bar d\,)}
{D-\bar d}=\,O(1+|x_n|);
\end{align*}
in particular, $\alpha_n>0$, and crucially $\beta^n =o(\log n)$.
Consequently,  since $\mu_1\le 2m$, \eqref{sumphid>bard} yields
\begin{equation}\label{sumphid>bard,compact}
\sum_{\bold v\,:\,d_{\bold v}\ge \bar d}\phi(d_{\bold v})\le O\bigl(m(1+|x_n|)\bigr).
\end{equation}
Next, if $d_{v}\in [0,\bar d\,)$, then, by \eqref{bard=}
$$
\phi(d_{ v})\le 2d_v + p\bar d\, d_{ v}= \frac{5}{2}\, d_{v}.
$$
so
\begin{equation}\label{sumphid<bard}
\sum_{v\,:\,d_{v}\le \bar d\,}\phi(d_{\bold v})\le \frac{5}{2}\sum_{v\,:\,d_{v}\le \bar d\,}d_{v}\le 5m.
\end{equation}
Combining \eqref{sumphid>bard,compact} and \eqref{sumphid<bard} we obtain
\begin{equation}\label{totalsumphid<}
\sum_{v}\phi(d_{ v})\le O\bigl(m(1+|x_n|)\bigr).
\end{equation}
Now the number of summands in the sum on the RHS of \eqref{sumnumu,degrees} 
is, at most,  the total number of positive solutions of $\|\bold d\|=2m$, which is
\begin{equation}\label{poscomp}
\binom{2m-1}{\nu-1}\le_b\left(\frac{2em}{\nu}\right)^{\nu}=\exp\bigl[\nu\log(m/\nu)
+O(\nu)\bigr].
\end{equation}
Also
\begin{equation}\label{binomsmall}
\binom{n}{\nu}\le 2^n,\quad (2m-1)!!=O((2m/e)^{m}). 
\end{equation}
Combining \eqref{sumnumu,degrees}, \eqref{totalsumphid<}-\eqref{binomsmall},
we obtain 
\begin{multline}\label{sumnumuless}
\sum_{G:\, |V(C)|=\nu,\,|E(C)|=m}\exp\bigl[-pX(G)\bigr]\\
\le\, \exp\bigl[-pnm+m\log m +\nu\log(m/\nu) +O(m(1 +|x_n|))\bigr].
\end{multline}
Here, since $m\le m_n:=3 n^{4/3}e^{|x_n|}$, 
\begin{multline*}
-pnm+m\log m +\nu\log(m/\nu)=-m\left(pn-\log m-\frac{\nu}{m}\log\frac{m}{\nu}\right)\\
\le -m\left(2\log n +2|x_n| - \frac{4}{3}\log n -\log 3 +O(n^{-1/3}\log n)\right)\le -\beta m\log n,
\end{multline*}
for a $\beta\in (0,2/3)$ and all large enough $n$. So, for $\beta^*\in (0,\beta)$,
\begin{equation}\label{sumnumu<beta*} 
\sum_{G:\, |V(C)|=\nu,\,|E(C)|=m}\exp\bigl[-pX(G)\bigr]\le \exp\bigl[-\beta^*m\log n\bigr], 
\end{equation}
uniformly for $\nu\ge an$,  $\nu-1\le m\le m_n$, and $n$ large enough.
It follows from \eqref{sumnumu<beta*} that
\begin{equation}\label{modernu,mu}
\sum_{\nu\ge an\atop \nu-1\le m\le m_n}\sum_{G:\, |V(C)|=\nu,\,|E(C)|=m}\exp\bigl[-pX(G)\bigr]
\le_b \exp(-\gamma n\log n),
\end{equation}
for $\gamma \in (0, a\beta^*)$.
\si

Combining \eqref{3lenulean}, \eqref{nu=2}, \eqref{mularge} and \eqref{modernu,mu} 
we complete the proof of Lemma \ref{p=(2logn+xn)/n}.\qed
\bi

Thus, for $p=(2 \log n +x_n)/n$, and $|x_n|=o(\log n)$, whp there are no extremal non-trivial
cocycles with support size $2$ or more. Let ${\Cal C}_n$ be the total number of  cocycles  with support size $1$, i.e.\ isolated edges.  By part (2) of this Lemma, if $x_n=c$, then 
$$
\lim_{n\to\infty}\Bbb E[{\Cal C}_n]=\lambda:=\frac{e^{-c}}{2}.
$$
And, as we had mentioned in the proof of Theorem \ref{thm:conn,k},  it can be shown that, in general,
$$
\lim_{n\to\infty}\Bbb E\bigl[({\Cal C}_n)_t\bigr]=\lambda^t,\quad t\ge 1.
$$
So ${\Cal C}_n$ is in the limit Poisson($\lambda$), which completes the proof of Theorem
\ref{thm:homconn,2}.\qed

\section{Proof of Theorem \ref{thm:homBT,2} }

As in the proof of Theorem \ref{BT,k}, we embed the $2$-dimensional process
$\{Y(n,M)\}$ into the continuous-time Markov process $\{Y_t(n)\}$. We introduce
$$
p_1=\frac{2\log n-\log\log n}{n},\quad p_2=\frac{2\log n+\log\log n}{n},
$$
and $t_i$ defined by $p_i=1-e^{-t_i}$.  For $k=2$, our argument showed that w.h.p.\ throughout
$[t_1,t_2]$ the complex $Y_t(n)$ continues to be a giant component plus a set of isolated
edges, whose number can decrease by $1$ only. That is, w.h.p.\ there is a random moment
$\tau\in (t_1,t_2)$ when the last isolated edge disappears, being swallowed by the giant
component. We need to show that w.h.p.\ $H^1(Y_{\tau}(n))$ vanishes as well. Suppose not.
Then there exists a $1$-cochain $f$, of support size $2$ or more, meeting the three 
conditions necessary for a non-trivial cocycle, such that none of
$X(G(f))$ triangles is present in $Y_{\tau-}(n)$, i.e.\   $Y_{\tau}(n)$ minus the triangle
born at time $\tau$. Let $(u,v)$ denote a generic value of the 
last isolated edge that disappeared at time $\tau$. Then 
\begin{equation}\label{XuvG,XG}
X_{(u,v)}(G(f))\le X(G(f))\le X_{(u,v)}(G(f)) + n-2,
\end{equation}
where $X_{(u,v)}(G(f))$ is the number of triangles, except those containing $(u,v)$, that
contain an odd number of edges of $G(f)$. Introduce $Y_t(n;(u,v))$, a subcomplex of
$Y_t(n)$, with all the triangles in $Y_t(n)$, if any, that contain $(u,v)$ being deleted. Then
\begin{multline}\label{PH1not}
\Bbb P(H^1(Y_{\tau}(n))\text{ does not vanish})\\
\le o(1)+\sum_{f,\,(u,v)}\int\limits_{t_1}^{t_2}(1-p(t))^{X_{(u,v)}(G(f))}\,(n-2)e^{-t}(1-p(t))^{n-3}\,dt.
\end{multline}
{\it Explanation.\/}  $o(1)$ stands for the probability that $\tau\in (t_1,t_2)$ does not exist.
$t$ is the generic value of a time when the edge $(u,v)$ stops being isolated. 
$(1-p(t))^{X_{(u,v)}(G(f))}$ is the probability that none of $X_{(u,v)}(G(f))$ triangles are
present in $Y_t(n;(u,v))$.
\si

Let $m=m(f)\ge 2$ be support size for $f$. By Meshulam-Wallach inequality and \eqref{XuvG,XG},
$$
X_{(u,v)}(G(f))\ge X(G(f)) - (n-2)\ge \frac{nm}{3} - (n-2).
$$
So, like part {\bf (2)\/} of the proof of Lemma \ref{p=(2logn+xn)/n}, for $m\ge 2m_n=6n^{4/3}e^{|c|}$, 
\begin{equation}\label{m>2mn}
\sum_{f:e(G(f))\ge 2m_n}(1-p(t))^{X_{(u,v)}(G(f))}\le \exp(-n^{4/3}),
\end{equation}
uniformly for $t\in [t_1,t_2]$. Hence the contribution of all such $f$'s to the RHS of 
\eqref{PH1not} is superexponentially small. 
\si

Suppose that $m \le 2m_n$. Let $\nu=|V(C(f))|$, $C(f):=C(G(f))$. Then
$$
X_{(u,v)}(G(f))\ge (\nu-1)m -\sum_{v\in V(C(f))}d_v^2.
$$
Indeed, using the proof of \eqref{Y13>} we see the following. (1) If $(w,x)$ from the support of
$f$ is not  in a triangle containing $(u,v)$, then the number of triangles that contain
$(w,x)$ as the only edge supporting $f$ is $\nu-(d_w+d_x)$, at least.  (2) If on the other hand
$(w,x)$ and $(u,v)$ are edges from the same triangle, then this triangle is unique, and so
the number of the triangles containing $(w,x)$ as the only edge supporting $f$  is
$(\nu-1)-(d_w+d_x)$, at least. Hence
\begin{align*}
X_{(u,v)}(G(f))\ge&\, (\nu-1)m -\sum_{(w,x)\in E(C(f))}\!\!\!(d_w+d_x)\\
=&\,(\nu-1)m -\sum_{v\in V(C(f))}\!\!\!d_v^2.
\end{align*}
So, with only trivial changes, we obtain a counterpart of \eqref{modernu,mu}: for $a>0$,
\begin{equation}\label{m,nu;moderate}
\sum_{\nu\ge an\atop \nu-1\le m\le 2m_n}\sum_{|V(C(f))|=\nu,\,|E(C(f))|=m}\!\!\!\!\!\!\!
(1-p(t))^{X_{(u,v)}(G(f))}
\le_b \exp(-\gamma n\log n),
\end{equation}
$\gamma=\gamma(a)>0$.
\si

Finally, suppose that $\nu\le an$, $a<1$. A counterpart of the bound \eqref{triv} is
$$
X_{(u,v)}(G(f))\ge (n-\nu-1)m,
$$
since, given an edge supporting $f$,  there can be at most one triangle that contains
this edge and $(u,v)$. So, with only minor changes in the part {\bf (1)\/} of Lemma
\ref{p=(2logn+xn)/n}, we obtain: for $a<1/2$,
\begin{multline}\label{nu<an}
\sum_{3\le \nu\le an\atop \nu-1\le m}\sum_{|V(C(f))|=\nu,\,|E(C(f))|=m}\!\!\!\!\!\!\!
(1-p(t))^{X_{(u,v)}(G(f))}\\
=O(n^{-1}e^{2\log\log n})=O(n^{-1}\log^2 n),
\end{multline}
uniformly for $t\in [t_1,t_2]$.
\bi

Combining \eqref{m>2mn}, \eqref{m,nu;moderate} and \eqref{nu<an}, we have
$$
\sum_f(1-p(t))^{X_{(u,v)}(G(f))} = O(n^{-1}\log^2 n).
$$
Consequently, the bound \eqref{PH1not} becomes
\begin{multline*}
\Bbb P(H^1(Y_{\tau}(n))\text{ does not vanish})\\
\le_b o(1)+n^3e^{-np_1}(t_2-t_1)n^{-1}\log^2 n\\
\le_b o(1)+n^{-1}\log^4 n\to 0,\quad n\to\infty.
\end{multline*}
This completes the proof of Theorem \ref{thm:homBT,2}.\qed

\section{Proof of Theorem \ref{thm:homconn,2} for $k\ge 2$} 

Our proof is a minor refinement of the proofs in \cite{lm}-\cite{mw}. Like the case $k=2$ in Section 3,
we need to show that, for 
$$
p=\frac{k \log n +x_n}{n},\quad |x_n|=o(\log n),
$$
\begin{align}
\sum_{H:|E(H)|>1}\exp\bigl[-pX(H)\bigr]\to 0,\label{3.5}\\
\sum_{H:|E(H)|=1}\exp\bigl[-pX(H)\bigr]\sim\frac{e^{-x_n}}{k!}.\label{3.6}
\end{align}
Here $H$ is a hypergraph on hypervertex-set $\binom{[n]}{[k-1]}$, with hyperedge-set 
$E(H)\subseteq \binom{[n]}{[k]}$, and $X(H)$ is the total number of $k$-simplexes that contain
an odd number of hyperedges of $H$. Further, an admissible $H$ meets the conditions:
\si
(1) maximum hypervertex degree is at most $\lfloor (n-k+1)\rfloor/2$;
\si
(2) $H$ has a single non-trivial component;
\si
(3) $X(H)\ge n|E(H)|/(k+1)$.
\si

For $|E(H)|$  ``small'',  the argument is analogous to the part {\bf (1)\/} of the proof
of Lemma \ref{p=(2logn+xn)/n}. For completeness, here it is.

Let $\nu=\nu(H)$ denote the range of $H$, i. e. 
$$
\nu(H):=|\{i\in [n]: i\in e\text{ for some }e\in E(H)\}|;
$$
then $X(H)\ge (n-\nu)m$. Since $H$ has a single non-trivial component, 
$$
m:=|E(H)|\ge \nu - (k-1).
$$

Let $k\le\nu\le an$, $0<a<1/(2k)$. Then $X(H)\ge (n-\nu)m$. The total number
of $H$ with $|E(H)|=m$ and $\nu(H)=\nu$, is at most
$$
\binom{n}{\nu}\binom{\binom{\nu}{k}}{m}\le \binom{n}{\nu}\left(\frac{e\binom{\nu}{k}}{m}\right)^m.
$$
So
$$
\sum_{H:\nu(H)=\nu,\,|E(H)|=m}\exp\bigl[-pX(H)\bigr]\le\, S(\nu,m),
$$
where
$$
S(\nu,m):=\binom{n}{\nu} \left(\frac{e\binom{\nu}{k}}{m}\right)^{m}\exp\bigl[-p(n-\nu)m\bigr].
$$
Now, using $m\ge \nu -(k-1)$,
\begin{align*}
\frac{S(\nu,m+1)}{S(\nu,m)}\le&\, \frac{\nu^2}{m+1}\exp\bigl[-p(n-\nu)\bigr]\\
\le&\, k\nu\exp\bigl[-p(n-\nu)\bigr]\le_b \frac{e^{|x_n|}}{n^{(k-1)-ka}}\to 0,
\end{align*}
as $k-1-ka>0$. Hence
\begin{align*}
\sum_{m\ge\nu-1}S(\nu,m)\le& 2S(\nu,\nu-k+1)\le_b S(\nu);\\
S(\nu):=&\,(\beta n)^{\nu}\nu^{(k-2)\nu}\exp\bigl[-p(n-\nu)(\nu-k+1)\bigr],
\end{align*}
for some $\beta>0$. It is easy to check that $\rho(\nu):=S(\nu+1)/S(\nu)$ increases with $\nu$,
and
$$
\rho(an)\le_b\frac{e^{|x_n|}}{n^{1-2ak}}\to 0,
$$
as $a<1/2k$. So
\begin{equation*}
\sum_{\nu=k+1}^{an}S(\nu)\le_b S(k+1)\le_b \frac{e^{2|x_n|}}{n^{k-1}},
\end{equation*}
which proves that
$$
\sum_{H: |E(H)|>1,\,\nu(H)\le an}\exp\bigl[-pX(H)\bigr]\to 0.
$$
In addition,
$$
\sum_{H:|E(H)|=1}\exp\bigl[-pX(H)\bigr]=\binom{n}{k}\exp\bigl[-p(n-k)]\sim\frac{e^{-x_n}}{k!},
$$
which proves \eqref{3.6}.
\si

The relation \eqref{3.5} will follow if we can show that 
$$
\sum_{H: |E(H)|>1,\,\nu(H)\ge  an}\exp\bigl[-pX(H)\bigr]\to 0,
$$
as well. For those $H$,
$$
|E(H)|\ge \nu(H) +k-1\ge an+k-1\ge an/2,
$$
and all we need is to cite a remarkable bound established in \cite{mw}: for a given
$\alpha>0$,
\begin{equation}\label{remark}
\sum_{H: |E(H)|\ge \alpha n}\exp\bigl[-pX(H)\bigr]\le \exp\bigl[-\Omega(n\log n)\bigr].
\end{equation}
(To be sure, \eqref{remark} was stated and proved for $x_n\to\infty$ arbitrarily slow. 
However, only obvious changes in the proof are required to cover $|x_n|=o(\log n)$.)
This highly non-trivial estimate was obtained by showing, via probabilistic method, that
for every hypergraph $H$, either there are ``many'' $k$-simplexes that contain exactly one
of hyperedges of $H$, or there is a ``small'' subset $\hat E\subset E(H)$ such that almost all other
hyperedges are incident to hyperedges in $\hat E$.
\bi

With \eqref{3.5}-\eqref{3.6}, the proof of Theorem \ref{thm:homconn,2} for $k>2$
is completed exactly like for $k=2$ in Section 3.\qed

\si

\bi

{\bf Acknowledgment.\/} The authors are very grateful to Nati Linial and Roy Meshulam for several 
inspiring, helpful conversations, both in person and through email. 

\bibliographystyle{plain}
\bibliography{KPrefs}
\end{document}